\documentclass[a4paper]{amsart}

\def\person{paule} % or coauthor. Then change input and bibliography directories.

\usepackage{tikz}

\def\mylabelonoff{off}
\def\allowdisbrkyesno{yes}
\def\numberingtheoremsectionyesno{no}
\def\numberingequationsectionyesno{no}
\def\pagesizeextendednormal{extended}

\def\reportudemathyesno{no}
\def\reportudemathnumber{SM-UDE-786}
\def\reportudemathyear{2015}
\def\reportudematheingang{\mydate}

\def\mytitle{On Korn's First Inequality\\ for Tangential or Normal Boundary Conditions with Explicit Constants}
\def\myshorttitle{On Korn's First Inequality}
\def\myauthorone{Sebastian Bauer}
\def\myauthortwo{Dirk Pauly}
\def\myauthors{\myauthorone\quad\&\quad\myauthortwo}
\def\myaddressone{Fakult\"at f\"ur Mathematik,
Universit\"at Duisburg-Essen, Campus Essen, Germany}

\def\myemailone{sebastian.bauer.seuberlich@uni-due.de}
\def\myemailtwo{dirk.pauly@uni-due.de}
\def\mykeywords{Korn inequality, tangential and normal boundary conditions, Boltzmann equation}
\def\mysubjclass{49J40 / 82C40 / 76P05}
\def\mydate{\today}
\usepackage{ifthen}
\ifthenelse{\equal{\person}{paule}}
{
\usepackage[mathscr]{eucal}
\usepackage[english]{babel}
\usepackage{a4,exscale,ifthen,amsfonts,amssymb,amsmath,amscd,graphicx,color}
\usepackage{nicefrac,tikz,fancyhdr,caption,array,multirow,multicol,booktabs,algorithm,algorithmic}
\usepackage[all]{xy}

\ifthenelse{\equal{\mylabelonoff}{on}}
{\newcommand{\mylabel}[1]{\label{#1}\fbox{{\sf #1}}}}
{\newcommand{\mylabel}[1]{\label{#1}}}
\ifthenelse{\equal{\allowdisbrkyesno}{yes}}
{\allowdisplaybreaks}
{}

\ifthenelse{\equal{\pagesizeextendednormal}{extended}}
{\setlength{\textwidth}{16cm}
\setlength{\textheight}{22cm}
\setlength{\oddsidemargin}{-0.2cm}
\setlength{\evensidemargin}{-0.2cm}}
{}
\ifthenelse{\equal{\numberingequationsectionyesno}{yes}}
{\numberwithin{equation}{section}}
{}

\newcommand{\diss}{\displaystyle}

\newcommand{\ovl}[1]{\overline{#1}}

\newcommand{\qtext}[1]{\quad\text{#1}\quad}

\newcommand{\cp}{c_{\mathsf{p}}}

\newcommand{\set}[2]{\{#1\,:\,#2\}}
\newcommand{\setb}[2]{\big\{#1\,:\,#2\big\}}

\ifthenelse{\equal{\numberingtheoremsectionyesno}{yes}}
{\newtheorem{lem}{Lemma}[section]}
{\newtheorem{lem}{Lemma}}
\newtheorem{defi}[lem]{Definition}
\newtheorem{theo}[lem]{Theorem}
\newtheorem{cor}[lem]{Corollary}
\newtheorem{rem}[lem]{Remark}
\newtheorem{pro}[lem]{Proposition}

\newtheorem{ex}[lem]{Example}

\newcommand{\om}{\Omega}
\newcommand{\omb}{\ovl{\om}}

\newcommand{\ga}{\Gamma}
\newcommand{\gat}{\ga_{\mathsf{t}}}
\newcommand{\gan}{\ga_{\mathsf{n}}}

\newcommand{\eps}{\epsilon}

\newcommand{\calR}{\mathcal{R}}

\newcommand{\reals}{\mathbb{R}}
\newcommand{\n}{\mathbb{N}}

\newcommand{\rtwo}{\reals^{2}}
\newcommand{\rt}{\reals^{3}}
\newcommand{\rN}{\reals^{N}}

\newcommand{\rNtN}{\reals^{N\times N}}

\newcommand{\foh}{\frac{1}{2}}
\newcommand{\oh}{\nicefrac{1}{2}}
\newcommand{\moh}{-\oh}

\newcommand{\equi}{\Leftrightarrow}
\newcommand{\qequi}{\quad\equi\quad}

\DeclareMathOperator{\id}{id}
\DeclareMathOperator{\sym}{sym}
\DeclareMathOperator{\skw}{skw}

\DeclareMathOperator{\dev}{dev}
\DeclareMathOperator{\tr}{tr}

\newcommand{\tcomp}[1]{#1_{\mathsf{t}}}
\newcommand{\ncomp}[1]{#1_{\mathsf{n}}}

\DeclareMathOperator{\p}{\partial}
\DeclareMathOperator{\na}{\nabla}

\DeclareMathOperator{\rot}{rot}

\DeclareMathOperator{\divergence}{div}
\renewcommand{\div}{\divergence}

\DeclareMathOperator{\ed}{d}

\newcommand{\csymbol}{\mathsf{C}}
\newcommand{\cgen}[3]{\overset{#1}{\csymbol}{}^{#2}_{#3}}

\newcommand{\ci}{\cgen{}{\infty}{}}
\newcommand{\cic}{\cgen{\circ}{\infty}{}}

\newcommand{\cict}{\cgen{\circ}{\infty}{\mathsf{t}}}
\newcommand{\cicn}{\cgen{\circ}{\infty}{\mathsf{n}}}
\newcommand{\cictn}{\cgen{\circ}{\infty}{\mathsf{t,n}}}

\newcommand{\co}{\cgen{}{1}{}}

\newcommand{\coo}{\cgen{}{1,1}{}}

\newcommand{\ct}{\cgen{}{2}{}}

\newcommand{\cicrN}{\cic(\rN)}

\newcommand{\cicom}{\cic(\om)}

\newcommand{\cictom}{\cict(\om)}
\newcommand{\cicnom}{\cicn(\om)}
\newcommand{\cictnom}{\cictn(\om)}

\newcommand{\cicomb}{\cic(\omb)}

\newcommand{\lsymbol}{\mathsf{L}}
\newcommand{\lgen}[3]{\overset{#1}{\lsymbol}{}^{#2}_{#3}}

\newcommand{\lt}{\lgen{}{2}{}}

\newcommand{\ltmo}{\lgen{}{2}{-1}}

\newcommand{\ltloc}{\lgen{}{2}{\mathsf{loc}}}

\newcommand{\ltom}{\lt(\om)}

\newcommand{\ltmoom}{\ltmo(\om)}

\newcommand{\ltlocom}{\ltloc(\om)}

\newcommand{\hsymbol}{\mathsf{H}}

\newcommand{\hgen}[3]{\overset{#1}{\hsymbol}{}^{#2}_{#3}}

\newcommand{\ho}{\hgen{}{1}{}}
\newcommand{\htwo}{\hgen{}{2}{}}

\newcommand{\hoc}{\hgen{\circ}{1}{}}

\newcommand{\hoct}{\hgen{\circ}{1}{\mathsf{t}}}
\newcommand{\hocn}{\hgen{\circ}{1}{\mathsf{n}}}
\newcommand{\hoctn}{\hgen{\circ}{1}{\mathsf{t,n}}}

\newcommand{\homo}{\hgen{}{1}{-1}}

\newcommand{\homoc}{\hgen{\circ}{1}{-1}}

\newcommand{\homoctn}{\hgen{\circ}{1}{-1,\mathsf{t,n}}}

\newcommand{\hoom}{\ho(\om)}

\newcommand{\htom}{\htwo(\om)}

\newcommand{\hocom}{\hoc(\om)}

\newcommand{\hoctom}{\hoct(\om)}
\newcommand{\hocnom}{\hocn(\om)}
\newcommand{\hoctnom}{\hoctn(\om)}
\newcommand{\homoctnom}{\homoctn(\om)}

\newcommand{\homoom}{\homo(\om)}

\newcommand{\homocom}{\homoc(\om)}

\newcommand{\norm}[1]{|#1|}

\newcommand{\normltom}[1]{\norm{#1}_{\ltom}}

\newcommand{\normltmoom}[1]{\norm{#1}_{\ltmoom}}

\newcommand{\scp}[2]{\langle#1,#2\rangle}

\newcommand{\scpltom}[2]{\scp{#1}{#2}_{\ltom}}

\newcommand{\bthreemat}[9]{\begin{bmatrix}#1&#2&#3\\#4&#5&#6\\#7&#8&#9\end{bmatrix}}

\newcommand{\preprintudemath}[5]{
\thispagestyle{empty}
\Large
\begin{center}SCHRIFTENREIHE DER FAKULT\"AT F\"UR MATHEMATIK\end{center}
\vspace*{5mm}
\begin{center}#1\end{center}
\vspace*{5mm}
\begin{center}by\end{center}
\begin{center}#2\end{center}
\vspace*{5mm}
\begin{center}#3\hspace{80mm}#4\end{center}
\newpage
\thispagestyle{empty}
\vspace*{210mm}
Received: #5
\newpage
\addtocounter{page}{-2}
\normalsize}
}
{}

\title[\sc\myshorttitle]{\Large\sf\mytitle}
\author{\myauthorone}
\author{\myauthortwo}
\address{\myaddressone}
\email[\myauthorone]{\myemailone}
\email[\myauthortwo]{\myemailtwo}
\subjclass{\mysubjclass}
\date{\mydate}
\begin{document}
%%%%%%%%%%%%%%%%%%%%%%%%%%%%%%%%%%%%%%%%%%%%%%%%%%%%%%%%%%%%%%%%%%%%%%%%%%%%%%%%

%%%%%%%%%%%%%%%%%%%%%%%%%%%%%%%%
% report series Duisburg-Essen %
%%%%%%%%%%%%%%%%%%%%%%%%%%%%%%%%

\ifthenelse{\equal{\reportudemathyesno}{yes}}
{\preprintudemath{\mytitle}{\myauthors}{\reportudemathnumber}{\reportudemathyear}{\reportudematheingang}}
{}

%%%%%%%%%%%%%%%%%%%%%%%%%%%%%%%%%%%%%%%%%%%%%%%%%%%%%%%%%%%%%%%%%%%%%%%%%%%%%%%%

\begin{abstract}
We will prove that for piecewise $\ct$-concave domains in $\rN$
Korn's first inequality holds for vector fields satisfying
homogeneous normal or tangential boundary conditions
with explicit Korn constant $\sqrt{2}$.
\end{abstract}

\maketitle
\tableofcontents

%%%%%%%%%%%%%%%%%%%%%%%%%%%%%%%%%%%%%%%%%%%%%%%%%%%%%%%%%%%%%%%%%%%%%%%%%%%%%%%%

\section{Introduction}

In \cite{desvillettesvillanikornnormal}, Desvillettes and Villani proved
a non-standard version of Korn's first inequality 
\begin{align}
\mylabel{firstkornnormal}
\normltom{\na v}
\leq c_{\mathsf{k,n}}\normltom{\sym\na v}
\end{align}
on non-axisymmetric sufficiently smooth bounded domains in $\rN$ for vector fields being tangential at the boundary.
Here $c_{\mathsf{k,n}}>0$ denotes the best available constant and the indices 
$\mathsf{k}$, $\mathsf{n}$ refer to 'Korn' and 'homogenous normal boundary condition'.
As pointed out in \cite{desvillettesvillanitrendglequiboltzmann},
this Korn inequality has an important application 
in statistical physics, more precisely in the study of relaxation to equilibrium of rarefied gases 
modeled by Boltzmann's equation.

In the paper at hand, we will show that for piecewise $\ct$-domains\footnote{Throughout this contribution,
$\ct$ can always be replaced by $\coo$.}
in $\rN$ with concave or even polyhedral boundary parts (see Definition \ref{admdom})
Korn's first inequality holds for vector fields satisfying (possibly mixed)
homogeneous normal or homogenous tangential boundary conditions, 
see \eqref{CtCn_def} and \eqref{Ctn_def} for a definition of the relevant spaces.
In every case the Korn constant can be estimated by $\sqrt{2}$: 
$$\normltom{\na v}
\leq\sqrt{2}\normltom{\sym\na v}\,,$$
see Theorem \ref{maintheo} for a precise statement.
The proof of our main theorem consists of a simple combination of two pointwise
equalities of the gradient of a vector field, see \eqref{graddevsymskw} and \eqref{rotgrad}, 
and an integration by parts formula
derived e.g. by Grisvard in \cite[Theorem 3.1.1.2]{grisvardbook}, see Proposition  \ref{GrisvardPro}. 
But before going into details of the proof we shall discuss
some disturbing consequences of  Theorem \ref{maintheo} seriously questioning 
at least the physical justification of full normal boundary conditions.

It is well known that Korn's first inequality with full normal boundary condition 
does not hold if $\om$ is axisymmetric.
We illustrate this fact with a simple example:
Let $\om\subset\rt$ be a bounded body of rotation with axis of symmetry $x_{1}=x_{2}=0$, e.g.,
$\om$ could be a ball, a cylinder or a cone.
Then the vector field $v$ defined by $v(x):=(x_{2},-x_{1},0)^\top$ belongs to $\hoom$ 
and is tangential to $\partial\om$. Hence, $v\in\hocnom$ (for a precise definition of $\hocnom$ see \eqref{CtCn_def}) and
$$\na v=\bthreemat{0}{-1}{0}{1}{0}{0}{0}{0}{0},\quad
\sym\na v=0,\quad
\div v=0.$$
Thus, Korn's first inequality with full normal boundary condition, see \eqref{firstkornnormal},
fails for these special domains $\om$, i.e., $c_{\mathsf{k,n}}=c_{\mathsf{k,n}}(\om)=\infty$.
On the other hand, Theorem \ref{maintheo} applies 
for every polyhedral approximation $\om_{\mathsf{p}}$ of $\om$ and we have
$$\forall\,v\in\hocn(\om_{\mathsf{p}})\qquad
\norm{\na v}_{\lt(\om_{\mathsf{p}})}
\leq\sqrt{2}\norm{\sym\na v}_{\lt(\om_{\mathsf{p}})}.$$
This means the (first) Korn constant can jump from $\sqrt{2}$ to $\infty$
caused by an arbitrary small deformation of the domain.
Many numerical schemes work on polyhedral domains of computation.
The Korn constants of all these domains are bounded from above by $\sqrt{2}$. 
But in many applications the domain of computation is just an approximation of
some `real' domain, whose Korn constant could be much larger.

Furthermore we shall discuss some conjectures on the meaning of Korn's first constant made in
\cite{desvillettesvillanikornnormal} and \cite{desvillettesvillanitrendglequiboltzmann}.
In \cite{desvillettesvillanikornnormal} Korn's first inequality with normal boundary condition, 
i.e. \eqref{firstkornnormal}, is proved 
for a bounded $\co$-domain $\om\subset\rN$ which is not axisymmetric.
An upper bound for the first Korn constant is presented by\footnote{In \cite{desvillettesvillanikornnormal}
the notations are different. For the constants we have
$c_{\mathsf{k,n}}^{-2}=K(\om)$, $c_{\mathsf{k}}^{-2}=\ovl{K}(\om)$,
$c_{\mathsf{m,n}}^2=C_{H}(\om)$ and $c_{\mathsf{g}}=G(\om)$.}
\begin{align}
\mylabel{kornconstrelation}
c_{\mathsf{k,n}}^2
&\leq2N(1+c_{\mathsf{m,n}}^2)(1+c_{\mathsf{k}}^2)(1+c_{\mathsf{g}}^{-1}),
\end{align}
where $c_{\mathsf{k}}$ denotes the first Korn constant 
for vector fields in $\hoom$ without boundary conditions,
$c_{\mathsf{m,n}}$ a special Gaffney constant 
for tangential vector fields in $\hoom$
and $c_{\mathsf{g}}$ the so called Grad's number defined by
$$c_{\mathsf{g}}
:=\frac{1}{|\om|}\inf_{|\sigma|=1}\inf_{v_{\sigma}\in\mathsf{V}_{\mathsf{n},\sigma}(\om)}\normltom{\sym\na v_{\sigma}}^2$$
with the finite dimensional set
$$\mathsf{V}_{\mathsf{n},\sigma}(\om)
:=\set{v\in\hocnom}{\div v=0\,\wedge\,\rot v=\sigma},\quad
\sigma\in S^{(N-1)N/2-1}.$$
For a precise definition of and more comments on these constants, see Section 4.

It is now conjectured in 
\cite[pages 607f]{desvillettesvillanikornnormal} and \cite[pages 285, 306 and (48)]{desvillettesvillanitrendglequiboltzmann}
that the constant $c_{\mathsf{k,n}}$ quantifies the deviation of 
$\om\subset\rN$ from being axisymmetric in the sense that $c_{\mathsf{k,n}}=c_{\mathsf{k,n}}(\om)$ tends to infinity,
if $\om$ is approaching axial symmetry. For such a statement it would be necessary to bound $c_{\mathsf{k,n}}$
from below while in  \cite{desvillettesvillanikornnormal} only the bound \eqref{kornconstrelation} from above is proved.
However,  Theorem \ref{maintheo} clearly shows that this 
conjecture becomes false at least if polyhedra are allowed to compete.
In \cite{desvillettesvillanikornnormal} and
\cite[page 609]{desvillettesvillanitrendglequiboltzmann}
it is also conjectured that it is Grad's number $c_{\mathsf{g}}=c_{\mathsf{g}}(\om)$
steering this blow-up of the Korn constant. 
It is conjectured that for smooth domains
Grad's number tends to zero while the domain $\om$ is approaching axial symmetry.
The following is actually stated in \cite[Proposition 5]{desvillettesvillanikornnormal}: 
$c_{\mathsf{g}}(\om)=0$ if and only if $\om$ is axisymmetric.
Moreover, there is a lower bound on $c_{\mathsf{g}}(\om)$ which
depends on the shape of $\om$. 
But in order to prove the conjecture it would 
be necessary to give an upper bound on $c_{\mathsf{g}}(\om)$ tending to zero if the domain is approaching axial sysmmetry.
However, this conjecture gets wrong, too,  if we allow for polyhedra: 
In Section 4 we will show that 
$$c_{\mathsf{g}}(\om_{\mathsf{p}})= \foh$$
holds for every convex bounded polyhedron $\om_{\mathsf{p}}$. 
Therefore, for any sequence of bounded and convex polyhedra tending to any axisymmetric domain Grad's number equals $\oh$.

The remaining part of the paper is organized as follows:
In Section 2 we give the relevant definitions on the spaces and domains used and
establish some equalities and inequalities used in the sequel.
In Section 3 we state our main theorem in detail and give the proof.
In Section 4 we discuss the constants $c_{\mathsf{k,n}}$, $c_{\mathsf{k}}$, $c_{\mathsf{m, n}}$ and $c_{\mathsf{g}}$
and give some further comments on the regularity of the boundary needed in the proof of 
\eqref{firstkornnormal} in \cite{desvillettesvillanikornnormal}.
In the last section we provide some more results estimating the gradient of a vector field. 

\section{Preliminaries}
\mylabel{sec:prelim}

Let $\om$ be an open subset of $\rN$
with $2\leq N\in\n$ and boundary $\ga:=\p\om$.
We introduce the standard scalar valued Lebesgue and Sobolev spaces by $\ltom$ and $\hoom$, respectively. 
Moreover, we define $\hocom$ as closure  in $\hoom$ 
of smooth and compactly supported test functions $\cicom$.
These definitions extend component-wise to vector or matrix fields
and we will use the same notations for these spaces throughout the paper. 
Moreover, we will consistently denote functions by $u$ and vector fields by $v$.
If $\om$ is Lipschitz, we define the vector valued Sobolev space $\hoctom$ resp. $\hocnom$ 
as closure in $\hoom$ of the set of test vector fields 
\begin{align}
\mylabel{CtCn_def}
\cictom
:=\setb{v\in\cicomb}{\tcomp{v}=0},\quad
\cicnom
:=\setb{v\in\cicomb}{\ncomp{v}=0},
\end{align}
respectively, generalizing homogeneous tangential resp. normal boundary conditions. 
Here, $\nu$ denotes the a.e. defined outer unit normal at $\ga$ giving a.e. 
the tangential resp. normal component
$$\tcomp{v}:=v|_{\ga}-\ncomp{v}\nu,\quad
\ncomp{v}:=\nu\cdot v|_{\ga}$$
of $v$ on $\ga$.
Here, we denote as usual 
$$\cicomb:=\setb{v|_{\om}}{v\in\cicrN}.$$
For smooth functions or vector fields $ v$ in $\hoom$ we have\footnote{The cross-product
notation needs an explanation. If we identify 
vector fields $a,b$ in $\rN$ with $1$-forms $\alpha,\beta$ we have for $|\beta|=1$ the identity
$\alpha=\beta\wedge*\beta\wedge*\alpha+(-1)^{N}*\beta\wedge*\beta\wedge\alpha$, 
where the wedge and Hodge star operations are executed from right to left. 
Especially in $\rt$ we have
$*\beta\wedge\alpha\cong b\times a$, $*\beta\wedge*\beta\wedge\alpha\cong b\times b\times a$ 
and $*\beta\wedge*\alpha\cong b\cdot a$. Hence,
$a=(b\cdot a)b-b\times b\times a$. For $b:=\nu$ and $a:=v|_{\ga}$
we get $\tcomp{v}=v|_{\ga}-\ncomp{v}\nu=-\nu\times\nu\times v|_{\ga}$ 
and we see $\tcomp{v}=0$ if and only if $\nu\times v|_{\ga}=0$.
Now, in this spirit the cross product in $\rN$ for vector fields is generally defined by 
$b\times a:\cong*\beta\wedge\alpha$, the latter being a $(N-2)$-form.
This yields e.g. $b\times a=b_{1}a_{2}-b_{2}a_{1}$ in $\rtwo$
or generally $b\times a\in\reals^{(N-1)N/2}$ in $\rN$.}
$$v\in\hocom\qequi v|_{\ga}=0,\qquad
v\in\hoctom\qequi\nu\times v|_{\ga}=0,\qquad
v\in\hocnom\qequi\nu\cdot v|_{\ga}=0.$$
If $\ga$ is decomposed into two relatively open subsets
$\gat$ and $\gan:=\ga\setminus\ovl{\gat}$ we define 
the vector valued $\ho$-Sobolev space of mixed boundary conditions $\hoctnom$
as closure in $\hoom$ of the set of test vector fields
\begin{align}
\mylabel{Ctn_def}
\cictnom
:=\setb{v\in\cicomb}{\tcomp{v}|_{\gat}=0\,\wedge\,\ncomp{v}|_{\gan}=0},
\end{align}
generalizing
$\nu\times v|_{\gat}=0$ and $\nu\cdot v|_{\gan}=0$
for $v\in\hoctnom$, respectively. 
For matrices $A\in\rNtN$ we recall the notations
\begin{align*}
\sym A
&:=\foh(A+A^{\top}),&
\skw A
&:=\foh(A-A^{\top}),&
\dev A
&:=A-\id_{A},&
\id_{A}
&:=\frac{\tr A}{N} \id
\end{align*}
with $\tr A:=A\cdot\id$ using the pointwise scalar product. 
By pointwise orthogonality we have
$$\norm{A}^2
=\norm{\dev A}^2
+\frac{1}{N}\norm{\tr A}^2,\quad
\norm{A}^2
=\norm{\sym A}^2
+\norm{\skw A}^2,\quad
\norm{\sym A}^2
=\norm{\dev\sym A}^2
+\frac{1}{N}\norm{\tr A}^2$$
and hence $\norm{\dev A},N^{\moh}\norm{\tr A},\norm{\sym A},\norm{\skw A}\leq\norm{A}$.
Especially for $A:=\na v:=J_{v}^{\top}$, where $J_{v}$ denotes the Jacobian of $v\in\hoom$, 
we see pointwise a.e.\footnote{The $\rot$-operator can be defined as follows:
We identify smooth vector fields $a$ in $\rN$ with smooth $1$-forms $\alpha$.
Then $\rot a:\cong\ed\alpha$, where $\ed$ denotes the exterior derivative
and $\ed\alpha$ is a $2$-form. For $N=2$ we obtain the scalar valued rotation 
$\rot a=\p_{1}a_{2}-\p_{2}a_{1}$ and for $N=3$ the classical rotation $\rot a$
appears, whereas generally $\rot a(x)\in\reals^{(N-1)N/2}$ holds.}
\begin{align}
\nonumber
\norm{\skw\na v}^2
&=\foh\norm{\rot v}^2,\quad
\tr\na v=\div v
\intertext{and}
\mylabel{graddevsymskw}
\norm{\na v}^2
&=\norm{\dev\sym\na v}^2
+\frac{1}{N}\norm{\div v}^2
+\foh\norm{\rot v}^2.
\intertext{Moreover, we have}
\mylabel{rotgrad}
\norm{\na v}^2
&=\norm{\rot v}^2
+\scp{\na v}{(\na v)^{\top}},
\intertext{since}
\nonumber
2\norm{\skw\na v}^2
&=\foh\norm{\na v-(\na v)^{\top}}^2
=\norm{\na v}^2
-\scp{\na v}{(\na v)^{\top}}.
\end{align}

The simplest version of Korn's first inequality is the following.

\begin{lem}[Korn's first inequality: $\hoc$-version]
\mylabel{kornfirst}
For all $v\in\hocom$
$$\normltom{\na v}^2
=2\normltom{\dev\sym\na v}^2
+\frac{2-N}{N}\normltom{\div v}^2
\leq2\normltom{\dev\sym\na v}^2$$
and equality holds if and only if $\div v=0$ or $N=2$.
\end{lem}

Although the proof is very simple, we present it here,
since we will use the underlying idea later.

\begin{proof}
For all vector fields $v\in\cicom$ we have 
by\footnote{For smooth $1$-forms in $\rN$ we have
$-\Delta\alpha=\ed*\ed*\alpha+(-1)^{N}*\ed*\ed\alpha$.
This means for a corresponding smooth vector proxy in $\rN$ that
$-\Delta a=-\na\div a+\rot^{*}\rot a$, where $\rot^{*}\cong(-1)^{N}*\ed*$,
the latter mapping $2$-forms to $1$-forms,
denotes the (formal) adjoint of $\rot\cong\ed$. Hence, $\rot^{*}$
maps smooth vector fields in $\reals^{(N-1)N/2}$ to vector fields in $\rN$.
Especially in $\rt$ we have
$\rot^{*}=\rot$ and hence $-\Delta a=-\na\div a+\rot\rot a$.
In $\rtwo$ it holds $\rot^{*}=R\na$, where $R$ is the $90^{\circ}$-rotation matrix,
and hence $-\Delta a=-\na\div a+R\na\rot a$.} 
$-\Delta=\rot^{*}\rot-\na\div$
Gaffney's equality
\begin{align}
\mylabel{hocnarotdic}
\normltom{\na v}^2
&=\normltom{\rot v}^2+\normltom{\div v}^2,
\intertext{which extends to all $v\in\hocom$ by continuity. Hence, with \eqref{graddevsymskw}}
\mylabel{hocnasymskw}
\normltom{\na v}^2
&=\normltom{\dev\sym\na v}^2
+\foh\normltom{\na v}^2
+\frac{2-N}{2N}\normltom{\div v}^2
\end{align}
and the assertion follows immediately.
\end{proof}

Recalling that we work with exterior unit normals at the boundaries,
we now introduce our admissible domains. 

\begin{defi}
\mylabel{admdom}
We call $\om$ `piecewise $\ct$', if
\begin{itemize}
\item[(i)] 
$\ga$ is strongly Lipschitz, i.e., locally a graph of a Lipschitz function,
\item[(ii)] 
$\ga=\ga_0\cup\ga_1$, where $\ga_0$ has $(N-1)$-dimensional Lebesgue measure zero, 
$\ga_1$ is relatively open in $\ga$ and locally a graph of a $\ct$-function. 
\end{itemize}
We call $\om$ `piecewise $\ct$-convex' resp. `piecewise $\ct$-concave', if $\om$ is piecewise $\ct$ and
\begin{itemize}
\item[(iii)]  
the second fundamental form on $\ga_1$ induced by $\na\nu$ is positive resp. negative semi-definite.
\end{itemize}
\end{defi}

By assumptions the exterior unit normal $\nu$ can be extended into a neighborhood
of $\ga_{1}$ such that the second fundamental form, i.e., the gradient $\na\nu$,
and its trace $\tr\na\nu=\div\nu=2H$, where $H$ denotes the mean curvature,
are well defined. For precise definitions see e.g. \cite[Section 3.1.1]{grisvardbook}.

\begin{ex}
The following domains in $\rtwo$ are piecewise $\ct$-concave,
where the dotted lines indicate an exterior domain:
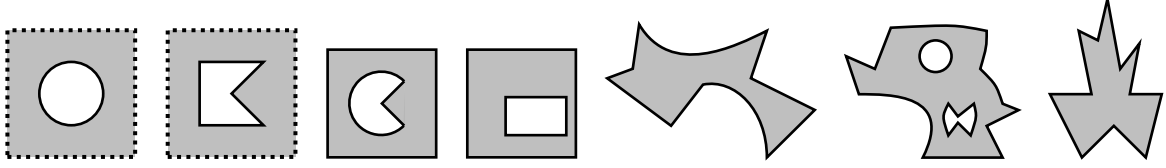
\def\scalepic{0.42}
\begin{center}
%\includegraphics[scale=\scalepic]{pic-ext-ball}
%\includegraphics[scale=\scalepic]{pic-ext-poly-1}
%\includegraphics[scale=\scalepic]{pic-pac-man}
%\includegraphics[scale=\scalepic]{pic-squares}
%\includegraphics[scale=\scalepic]{pic-star-concave}
%\includegraphics[scale=\scalepic]{pic-star-holes}
%\includegraphics[scale=\scalepic]{pic-star}
%%%%%%%%%%%%%%%%%%%%%%%%%%%%%%%%%%%%%%%%%%%%%%%%%%%%%%%%
% pics by Manu
%%%%%%%%%%%%%%%%%%%%%%%%%%%%%%%%%%%%%%%%%%%%%%%%%%%%%%%%
%========================================
% exterior domain: sphere
%========================================
\mbox{
\begin{tikzpicture}[scale=\scalepic]
\fill [gray,opacity=.5] (-2,-2) rectangle (2,2);				% space
\draw [dotted,line width=1.5pt] (-2,-2) rectangle (2,2);			% space boundary
\fill [white] (0,0) circle (1);								% circle
\draw [line width=1pt] (0,0) circle (1);						% circle boundary
\end{tikzpicture}
}
%========================================
% exterior domain: rectangular pac-man
%========================================
\mbox{
\begin{tikzpicture}[scale=\scalepic]
\fill [gray,opacity=.5] (-2,-2) rectangle (2,2);						% space
\draw [dotted,line width=1.5pt] (-2,-2) rectangle (2,2);					% space boundary
\fill [white] (-1,-1) -- (-1,1) -- (1,1) -- (0,0) -- (1,-1) -- cycle;				% pac-man
\draw [line width=1pt] (-1,-1) -- (-1,1) -- (1,1) -- (0,0) -- (1,-1) -- cycle;		% pac-man boundary
\end{tikzpicture}
}
%========================================
% rectangle \setminus normal pac-man
%========================================
\mbox{
\begin{tikzpicture}[scale=\scalepic]
\fill [gray,opacity=.5] (-1.7,-1.7) rectangle (1.7,1.7);				% domain fill
\draw [line width=1pt] (-1.7,-1.7) rectangle (1.7,1.7);				% domain outer boundary
\fill [white] (0.7,0.7) arc (44:315:1);							% pac-man fill
\draw [line width=1pt] (0.7,0.7) arc (44:315:1);					% pac-man boundary
\fill [gray,opacity=.5] (0.7,0.7) -- (0,0) -- (0.7,-0.7);				% pac-man mouth fill
\draw [line width=1pt] (0.7,0.7) -- (0,0) -- (0.7,-0.7);				% pac-man mouth boundary
\end{tikzpicture}
}
%========================================
% rectangle \setminus rectangle
%========================================
\mbox{
\begin{tikzpicture}[scale=\scalepic]
\fill [gray,opacity=.5] (-1.7,-1.7) rectangle (1.7,1.7);				% domain fill
\draw [line width=1pt] (-1.7,-1.7) rectangle (1.7,1.7);				% domain outer boundary
\fill [white] (-0.5,-1) rectangle (1.4,0.2);						% rectangle fill
\draw [line width=1pt] (-0.5,-1) rectangle (1.4,0.2);				% pac-man boundary
\end{tikzpicture}
}
%========================================
% random shape, no holes
%========================================
\mbox{
\begin{tikzpicture}[scale=\scalepic]
% domain fill
\fill [gray,opacity=.5] (0,0) -- (2,-1.5) -- (3,-0.2) .. controls (4,0) and (5,-1) .. (5,-2.5) -- (6.5,-1) -- (4.5,0) -- (5,1.5) .. controls (4,1) and (2,0) .. (1,1.7) -- (0.8,0.3) -- cycle;
% domain outer boundary
\draw [line width=1pt] (0,0) -- (2,-1.5) -- (3,-0.2) .. controls (4,0) and (5,-1) .. (5,-2.5) -- (6.5,-1) -- (4.5,0) -- (5,1.5) .. controls (4,1) and (2,0) .. (1,1.7) -- (0.8,0.3) -- cycle;
\end{tikzpicture}
}
%========================================
% random shape, two holes
%========================================
\mbox{
\begin{tikzpicture}[scale=\scalepic]
% domain fill
\fill [gray,opacity=.5] (0,0) .. controls (1,0) and (3,0) .. (2,-2) -- (4.5,-2) -- (4,-1) -- (5,-0.5) -- (4.5,-0.3) .. controls (4.3,0.3) ..(3.8,0.8) .. controls (4,1.5) .. (4,2) .. controls (3,2.2) .. (1,2.1) -- (0.5,0.8) -- (-0.4,1.2) -- cycle;
% domain outer boundary
\draw [line width=1pt] (0,0) .. controls (1,0) and (3,0) .. (2,-2) -- (4.5,-2) -- (4,-1) -- (5,-0.5) -- (4.5,-0.3) .. controls (4.3,0.3) ..(3.8,0.8) .. controls (4,1.5) .. (4,2) .. controls (3,2.2) .. (1,2.1) -- (0.5,0.8) -- (-0.4,1.2) -- cycle;
% \setminus cicle
\fill [white] (2.4,1.2) circle (0.5);
\draw [line width=1pt] (2.4,1.2) circle (0.5);
% \setminus non-convex part
\fill [white] (2.8,-0.3) .. controls (2.6,-0.7) .. (2.8,-1.3) -- (3.1,-0.9) -- (3.5,-1.3) .. controls (3.7,-0.7) .. (3.6,-0.3) -- (3.1,-0.7) -- cycle;
\draw [line width=1pt] (2.8,-0.3) .. controls (2.6,-0.7) .. (2.8,-1.3) -- (3.1,-0.9) -- (3.5,-1.3) .. controls (3.7,-0.7) .. (3.6,-0.3) -- (3.1,-0.7) -- cycle;
\end{tikzpicture}
}
%========================================
% starlike domain
%========================================
\mbox{
\begin{tikzpicture}[scale=\scalepic]
% domain fill
\fill [gray,opacity=.5] (0,0) -- (-1,-1) -- (-2,1) -- (-0.7,1) -- (-1.1,3) -- (-0.5,2.7) -- (-0.2,4) -- (0.2,1.8) -- (0.8,2.6) -- (0.5,1) -- (1.5,1) -- (1,-1) -- cycle;
% domain outer boundary
\draw [line width=1pt] (0,0) -- (-1,-1) -- (-2,1) -- (-0.7,1) -- (-1.1,3) -- (-0.5,2.7) -- (-0.2,4) -- (0.2,1.8) -- (0.8,2.6) -- (0.5,1) -- (1.5,1) -- (1,-1) -- cycle;
\end{tikzpicture}
}
%%%%%%%%%%%%%%%%%%%%%%%%%%%%%%%%%%%%%%%%%%%%%%%%%%%%%%%%
\end{center}
\end{ex}

Our main result is an easy consequence of the pointwise equalities \eqref{graddevsymskw} and \eqref{rotgrad}
and the following crucial proposition from Grisvard, \cite[Theorem 3.1.1.2]{grisvardbook}.

\begin{pro}[integration by parts]
\mylabel{GrisvardPro}
Let $\om$ be piecewise $\ct$. Then for all $v\in\cicomb$
\begin{align*}
\normltom{\div v}^2
-\scpltom{\na v}{(\na v)^{\top}}
&=\int_{\ga_1}\big(\div_\ga(\ncomp{v}\tcomp{v})
-2\tcomp{v}\cdot\na_\ga \ncomp{v}\big)
+\int_{\ga_1}\big(\div\nu\,|v_{\mathsf{n}}|^2
+((\na\nu)\,v_{\mathsf{t}})\cdot v_{\mathsf{t}}\big)
\intertext{and for $v\in\cictnom$}
\normltom{\div v}^2
-\scpltom{\na v}{(\na v)^{\top}}
&=\int_{\ga_1}\big(\div\nu\,|v_{\mathsf{n}}|^2
+((\na\nu)\,v_{\mathsf{t}})\cdot v_{\mathsf{t}}\big).
\end{align*}
\end{pro}

Here, $\div_\ga$ and $\na_\ga$ are the usual surface differential operators on $\ga_1$,
which may be identified with the co-derivative $*\ed*$ on $1$-forms
and the exterior derivative $\ed$ on $0$-forms on $\ga_{1}$, respectively.
Actually in \cite{grisvardbook} it is assumed that $\om$ is bounded.
But since we assume that $v$ has compact support, the asserted formulas
hold for unbounded domains as well. 

\begin{rem}
\mylabel{GrisvardProRem}
We note that in \cite{grisvardbook} Grisvard uses $-\na\nu$ 
to define the second fundamental form, which implies a negative sign for
the curvature term, i.e., the integral
$$-\int_{\ga_1}\big(\div\nu\,|v_{\mathsf{n}}|^2
+((\na\nu)\,v_{\mathsf{t}})\cdot v_{\mathsf{t}}\big)$$
appears in \cite[Theorem 3.1.1.2]{grisvardbook}.
Moreover, by 
$\div_\ga(\ncomp{v}\tcomp{v})
=\ncomp{v}\div_\ga\tcomp{v}+\tcomp{v}\cdot\na_\ga\ncomp{v}$ 
on $\ga_{1}$ we have 
$$\div_\ga(\ncomp{v}\tcomp{v})
-2\tcomp{v}\cdot\na_\ga \ncomp{v}
=\ncomp{v}\div_\ga\tcomp{v}
-\tcomp{v}\cdot\na_\ga \ncomp{v}.$$
\end{rem}

An immediate corollary of Proposition \ref{GrisvardPro} is the following.

\begin{cor}[Gaffney's inequalities]
\mylabel{concavepolyhedroncor}
Let $\om$ be piecewise $\ct$-convex resp. $\ct$-concave and $v\in\hoctnom$. Then
$$\normltom{\na v}^2
\leq\normltom{\rot v}^2+\normltom{\div v}^2
\qtext{resp.}
\normltom{\na v}^2
\geq\normltom{\rot v}^2+\normltom{\div v}^2.$$
If $\om$ is even a polyhedron, equality holds, i.e.,
\begin{equation}
\mylabel{polyhedroncor}
\normltom{\na v}^2
=\normltom{\rot v}^2+\normltom{\div v}^2.
\end{equation}
\end{cor}

\begin{proof}
By continuity it is sufficient to consider $ v\in\cictnom$ instead of $v\in\hoctnom$.
Using Proposition \ref{GrisvardPro} together with \eqref{rotgrad} we have 
\begin{align}
\mylabel{rotdivgradnanu}
\normltom{\div v}^2
+\normltom{\rot v}^2
=\normltom{\na v}^2
+\int_{\ga_1}\big(\div\nu\,|v_{\mathsf{n}}|^2
+((\na\nu)\,v_{\mathsf{t}})\cdot v_{\mathsf{t}}\big).
\end{align}
Due to the positive resp. negative semi-definiteness of the second fundamental form, 
the surface integral is non-negative resp. non-positive resp. vanishes.
\end{proof}

\begin{rem}
For $N=3$ formula \eqref{polyhedroncor} has already been proved in \cite[Theorem 4.1]{costabelcoercbilinMax}.
\end{rem}

\begin{rem}
\mylabel{approxrem}
By defining the Sobolev spaces with boundary conditions 
as closures of suitable test vector fields,
we avoid discussions about density or approximation arguments and properties.
We note that we do not claim
$$\hoctnom=\set{v\in\hoom}{\nu\times v|_{\gat}=0\,\wedge\,\nu\cdot v|_{\gan}=0},$$
although this equality seems to be reasonable.
On the other hand, it is known, that at least for polyhedra 
or curved polyhedra\footnote{In our notation, a so called curved polyhedron
has got a piecewise $\ci$-boundary.}
in $\rtwo$ or $\rt$ and either full tangential or full normal boundary condition
$$\hoctom=\set{v\in\hoom}{\nu\times v|_{\ga}=0},\quad
\hocnom=\set{v\in\hoom}{\nu\cdot v|_{\ga}=0}$$
hold, see \cite[Theorem 2.1, Lemma 2.6 ($J=1$)]{costabeldaugenicaisesingularitiesmaxwellinterface}
and for the curved case \cite[Theorem 2.3]{costabeldaugemaxwelllameeigenvaluespolyhedra}
and the corresponding proofs. To the best of the authors knowledge,
there are no proofs (yet) for general Lipschitz domains or mixed boundary conditions
showing these density properties.
We also want to point out that Proposition \ref{GrisvardPro}
and formula \eqref{rotdivgradnanu} (and Remark \ref{GrisvardProRem}) 
for the special case of $\om\subset\rt$
have been used e.g. in \cite[Lemma 2.1, Lemma 2.2]{costabeldaugemaxwelllameeigenvaluespolyhedra}
or \cite[Lemma 2.11]{amrouchebernardidaugegiraultvectorpot} as well.
\end{rem}

\section{Results}

\begin{theo}[Korn's first inequality: tangential/normal version]
\mylabel{maintheo}
Let $\om\subset\rN$ be piecewise $\ct$-concave and $v\in\hoctnom$. Then Korn's first inequality
$$\normltom{\na v}
\leq\sqrt{2}\normltom{\dev\sym\na v}$$
holds. If $\om$ is a polyhedron, even
$$\normltom{\na v}^2
=2\normltom{\dev\sym\na v}^2
+\frac{2-N}{N}\normltom{\div v}^2
\leq2\normltom{\dev\sym\na v}^2$$
is true and equality holds if and only if $\div v=0$ or $N=2$.
\end{theo}

\begin{proof}
We use \eqref{graddevsymskw} in combination with Corollary \ref{concavepolyhedroncor} to see
\begin{align*}
\normltom{\na v}^2
&\leq\normltom{\dev\sym\na v}^2
+\foh\normltom{\na v}^2
+\frac{2-N}{2N}\normltom{\div v}^2,
\end{align*}
which shows the first estimate.
If $\om$ is a polyhedron, we see by Corollary \ref{concavepolyhedroncor}
that equality holds in the latter estimate, which proves the other assertions.
\end{proof}

\begin{rem}[unbounded domains]
\mylabel{maintheorem}
All our results remain true for slightly weaker Sobolev spaces.
In exterior domains, i.e., domains with compact complement,
it is common to work in weighted Sobolev spaces like
\begin{align*}
\homoom
&:=\set{u\in\ltmoom}{\na u\in\ltom},\\
\ltmoom
&:=\set{u\in\ltlocom}{\rho^{-1}u\in\ltom},\quad
\rho:=(1+r^2)^{\oh},\quad
r(x):=|x|.
\end{align*}
If $N=2$ we have to replace $\ltmoom$ and $\homoom$ by $\lgen{}{2}{-1,\ln}(\om)$
and $\hgen{}{1}{-1,\ln}(\om)$, respectively,
where $u$ belongs to $\lgen{}{2}{-1,\ln}(\om)$ if $(\ln(e+r)\rho)^{-1}u\in\ltom$.
The Sobolev spaces generalizing the different boundary conditions
are defined as before as closures in $\homoom$ 
resp. $\hgen{}{1}{-1,\ln}(\om)$ of respective test functions.
For bounded domains, these weighted Sobolev spaces coincide with the standard ones
equipped with equivalent scalar products.
The reason for working in weighted Sobolev spaces 
is that the standard Poincar\'e inequalities do not hold in exterior domains.
As proper replacement we have weighted Poincar\'e inequalities, i.e., for $N\geq3$
$$\forall\,u\in\homoom\quad
\normltmoom{u}
\leq\cp\normltom{\na u},$$
and we note that the best Poincar\'e constant for $\homocom$ satisfies $\cp\leq2/(N-2)$,
see e.g. \cite[Poincar\'e's estimate III, p. 57]{leisbook},
\cite[Lemma 4.1]{saranenwitschexteriorell} or \cite[Appendix A.2]{paulyrepinell}.
Since our arguments only involve derivatives, it is clear that all our results,
mainly Theorem \ref{maintheo} but also the preceding lemmas and corollary,
extend easily to the family of Sobolev spaces in $\homoom$ resp. $\hgen{}{1}{-1,\ln}(\om)$.
\end{rem}

From the latter remark the following clearly holds true.

\begin{cor}[Korn's first inequality: weighted tangential/normal version]
\mylabel{maintheocor}
Theorem \ref{maintheo} extends to all $v$ in $\homoctnom$ 
resp. $\hgen{\circ}{1}{-1,\ln,\mathsf{t,n}}(\om)$ if $N=2$.
\end{cor}

\begin{cor}
\mylabel{maintheocorrotdiv}
Corollary \ref{concavepolyhedroncor} extends to all $v$ in $\homoctnom$ 
resp. $\hgen{\circ}{1}{-1,\ln,\mathsf{t,n}}(\om)$ if $N=2$.
Also Lemma \ref{kornfirst} extends to all $v$ in $\homocom$ 
resp. $\hgen{\circ}{1}{-1,\ln}(\om)$ if $N=2$.
\end{cor}

\section{Some Remarks on the Constants $c_{\mathsf{k}}$, $c_{\mathsf{m,n}}$ and $c_{\mathsf{g}}$}

In this section we want to discuss in detail some constants and inequalities used in 
\cite{desvillettesvillanikornnormal}.
In \cite{desvillettesvillanikornnormal} Korn's first inequality with normal boundary condition, 
i.e. \eqref{firstkornnormal}, is proved for a bounded $\co$-domain $\om\subset\rN$ 
which is not axisymmetric. As already mentioned in the introduction
an upper bound for the first Korn constant is presented by \eqref{kornconstrelation}, i.e.,
\begin{align*}
c_{\mathsf{k,n}}^2
&\leq2N(1+c_{\mathsf{m,n}}^2)(1+c_{\mathsf{k}}^2)(1+c_{\mathsf{g}}^{-1}),
\end{align*}
which we repeat here for the convenience of the reader.
All these constants depend on $\om$ and we always assume to deal with best possible ones.

\subsection{Korn Constant without Boundary Condition $c_{\mathsf{k}}$}

This constant belongs to the standard first Korn inequality without boundary conditions, i.e.,
\begin{align}
\mylabel{firstkornvillani}
\exists\,c_{\mathsf{k}}>0\qquad
\forall\,v\in\hoom\qquad
\exists\,r_{v}\in\calR\qquad
\normltom{\na(v-r_{v})}
\leq c_{\mathsf{k}}\normltom{\sym\na v},
\end{align}
where $\calR$ is the finite dimensional space of rigid motions
and $r_{v}$ the $\ltom$-orthonormal projection onto $\calR$.
Especially, \eqref{firstkornvillani} holds for any bounded Lipschitz domain $\om\subset\rN$.

\subsection{Normal Gaffney Constant $c_{\mathsf{m,n}}$}

Whereas the literature on $c_{\mathsf{k}}$ is well known, it seems that the knowledge
on the normal Gaffney constant $c_{\mathsf{m,n}}$ is more restricted to the community dealing 
with Maxwell's equations as it is explicitly noted in \cite{desvillettesvillanikornnormal}.
For this reason we examine it here in more detail.
In \cite{desvillettesvillanikornnormal} this constant appears in a special Gaffney inequality
for tangential vector fields in $\hoom$, i.e.,
there exists $c_{\mathsf{m,n}}>0$, such that 
for all $v\in\hocnom$
there exists $n_{v}\in\mathsf{V}_{\mathsf{n},0}(\om)$ with 
\begin{align}
\nonumber
\normltom{\na(v-n_{v})}
&\leq c_{\mathsf{m,n}}\big(\normltom{\skw\na v}^2+\normltom{\tr\na v}^2\big)^{\oh}
\intertext{or equivalently (with slightly different $c_{\mathsf{m,n}}$)}
\mylabel{maxwellregvillani}
\normltom{\na(v-n_{v})}
&\leq c_{\mathsf{m,n}}\big(\normltom{\rot v}^2+\normltom{\div v}^2\big)^{\oh},
\end{align}
where 
$$\mathsf{V}_{\mathsf{n},0}(\om)
:=\set{v\in\hocnom}{\div v=0\,\wedge\,\rot v=0}$$
is the finite dimensional\footnote{We remark 
$\mathsf{V}_{\mathsf{n},0}(\om)\subset\mathsf{X}_{\mathsf{n},0}(\om)$
and that even $\mathsf{X}_{\mathsf{n},0}(\om)$ is finite dimensional.} 
subspace of $\hoom$-Neumann fields
and $n_{v}$ the $\ltom$-orthonormal projection onto $\mathsf{V}_{\mathsf{n},0}(\om)$.
Inequality \eqref{maxwellregvillani} can be derived by a Maxwell regularity result,
see e.g. \cite{weberregmax,kuhnpaulyregmax}, stating the following: 
Let $$\mathsf{X}_{\mathsf{n}}(\om)
:=\set{v\in\ltom}{\rot v\in\ltom\,\wedge\,\div v\in\ltom\,\wedge\,\nu\cdot v|_{\ga}=0},$$
where the vanishing normal trace
has to be understood in the weak sense\footnote{The vanishing normal trace is 
realized by the closure of test vector fields $\cicom$
under the graph norm of $\div$ viewed as an unbounded operator acting on $\ltom$.}. 
If $\om\subset\rN$ is a bounded domain and either $\ct$ or convex, then any vector field $v$ in $\mathsf{X}_{\mathsf{n}}(\om)$
already belongs to $\hoom$, i.e., $v\in\hocnom$.
Since $\mathsf{X}_{\mathsf{n}}(\om)$ together with the norm 
$$\norm{v}_{\mathsf{X}_{\mathsf{n}}(\om)}^2
:=\normltom{v}^2+\normltom{\rot v}^2+\normltom{\div v}^2$$
is a Hilbert space, we can apply the closed graph theorem to the identity mapping $\mathsf{X}_{\mathsf{n}}(\om)$
to $\hoom$. Therefore there exists $c_{\mathsf{m,n,reg}}>0$, such that for all 
$v\in\mathsf{X}_{\mathsf{n}}(\om)\cap\hoom=\hocnom$
\begin{align}
\mylabel{maxwellregone}
\normltom{\na v}
&\leq c_{\mathsf{m,n,reg}}
\big(\normltom{v}^2+\normltom{\rot v}^2+\normltom{\div v}^2\big)^{\oh}
\end{align}
holds. Since the embedding of $\mathsf{X}_{\mathsf{n}}(\om)$ into $\ltom$ is compact
even for bounded Lipschitz (or weaker) domains $\om$,
see \cite{weckmax,webercompmax,picardcomimb,witschremmax,picardweckwitschxmas},
we also have the so-called normal Maxwell estimate, i.e.,
there exists $c_{\mathsf{m,n,est}}>0$, such that for all $v\in\mathsf{X}_{\mathsf{n}}(\om)$
there exists $n_{v}\in\mathsf{X}_{\mathsf{n},0}(\om)$ with
\begin{align}
\mylabel{maxwellest}
\normltom{v-n_{v}}
&\leq c_{\mathsf{m,n,est}}
\big(\normltom{\rot v}^2+\normltom{\div v}^2\big)^{\oh},
\end{align}
where 
$$\mathsf{X}_{\mathsf{n},0}(\om)
:=\set{v\in\mathsf{X}_{\mathsf{n}}(\om)}{\div v=0\,\wedge\,\rot v=0}$$
is the finite dimensional\footnote{We note that by the compact embedding 
$\mathsf{X}_{\mathsf{n}}(\om)\hookrightarrow\ltom$, the unit ball in $\mathsf{X}_{\mathsf{n},0}(\om)$
is compact.} 
subspace of Neumann fields
and $n_{v}$ the $\ltom$-orthonormal projection onto $\mathsf{X}_{\mathsf{n},0}(\om)$.
Now \eqref{maxwellregvillani} follows immediately by combining 
\eqref{maxwellregone} and \eqref{maxwellest} if $\om$ is bounded and either $\ct$ or convex with
$$c_{\mathsf{m,n}}\leq c_{\mathsf{m,n,reg}}\sqrt{c_{\mathsf{m,n,est}}^2+1},$$
since in this case $\mathsf{X}_{\mathsf{n}}(\om)=\hocnom$ and $\mathsf{X}_{\mathsf{n},0}(\om)=\mathsf{V}_{\mathsf{n},0}(\om)$.
In the bounded and convex case there are even no Neumann fields, i.e., $\mathsf{X}_{\mathsf{n},0}(\om)=\{0\}$,
and $c_{\mathsf{m,n}}\leq1$ holds, see e.g.
\cite{saranenineqfried,costabelcoercbilinMax,amrouchebernardidaugegiraultvectorpot,paulymaxconst0,paulymaxconst1,paulymaxconst2}
for the cases $N=2$ or $N=3$,
which follows essentially by Corollary \ref{concavepolyhedroncor} and uniform approximation 
of a convex domain $\om$ by a sequence of smooth and convex domains.
The Neumann fields generally vanish if and only if $\om$ is simply connected.
On the other hand we note that \eqref{maxwellregvillani} 
also holds in some non-smooth and non-convex situations as well.
For example, by Corollary \ref{appGrisvardCorwithoutltnorm} below we see that 
\eqref{maxwellregvillani} is valid if $\om$ is bounded and piecewise $\ct$.
Especially, for piecewise $\ct$-convex domains we have $c_{\mathsf{m,n}}\leq1$
by Corollary \ref{concavepolyhedroncor}. For polyhedra it even holds $c_{\mathsf{m,n}}=1$.
We note that in the latter piecewise $\ct$-convex case 
we can choose $n_{v}=0$ even if $\om$ is not simply connected, i.e.,
even if Neumann fields exist in $\mathsf{X}_{\mathsf{n}}(\om)$.
These possible Neumann fields must vanish by Corollary \ref{concavepolyhedroncor}
as soon as they belong to $\hocnom$.
Therefore there are domains, e.g. a polyhedron with a reentrant edge, 
where $\hocnom$ is a closed subspace of $\mathsf{X}_{\mathsf{n}}(\om)$ in the 
$\mathsf{X}_{\mathsf{n}}(\om)$-topology, 
but neither $\mathsf{X}_{\mathsf{n}}(\om)\not\subset\hocnom$
nor $\hocnom$ is dense in $\mathsf{X}_{\mathsf{n}}(\om)$.
To the best knowledge of the authors it is unknown, wether or not \eqref{maxwellregvillani} holds for
general bounded Lipschitz domains or even for general bounded $\co$-domains.

\subsection{Grad's Number $c_{\mathsf{g}}$ }

Let $\om\subset\rN$ be a bounded Lipschitz domain.
From the introduction we recall Grad's number
$$c_{\mathsf{g}}
=\frac{1}{|\om|}\inf_{|\sigma|=1}\inf_{v_{\sigma}\in\mathsf{V}_{\mathsf{n},\sigma}(\om)}\normltom{\sym\na v_{\sigma}}^2$$
and the finite dimensional set
$$\mathsf{V}_{\mathsf{n},\sigma}(\om)
=\set{v\in\hocnom}{\div v=0\,\wedge\,\rot v=\sigma},\quad
\sigma\in S^{(N-1)N/2-1}.$$
We emphasize that $\mathsf{V}_{\mathsf{n},\sigma}(\om)$ might be empty, 
if $\om$ is not smooth enough, since generally a solution of 
\begin{align}
\mylabel{Maxwellproblem}
\div v_{\sigma}=0,\,\rot v_{\sigma}=\sigma\text{ in }\om,\quad 
v_{\sigma}\cdot\nu=0\text{ on }\ga,
\end{align}
does not belong to $\hoom$. More precisely, \eqref{Maxwellproblem} admits a solution $v_{\sigma}$
$$v_{\sigma}\in\mathsf{X}_{\mathsf{n},\sigma}(\om)
:=\set{v\in\mathsf{X}_{\mathsf{n}}(\om)}{\div v=0\,\wedge\,\rot v=\sigma}$$
for any $\sigma\in S^{(N-1)N/2-1}$ since $\sigma$ belongs to the range of the rotation.
This follows by the simple fact that $\sigma=\rot\hat{\sigma}\in\rot\hoom$ 
or equivalently $\Sigma=-\skw\na\,\hat{\Sigma}\in\skw\na\hoom$ holds
with $\hat{\Sigma}(x):=\Sigma\,x$, where 
the skew-symmetric matrix $\Sigma\in\rNtN$ corresponds to 
$\sigma\in\reals^{(N-1)N/2}$ and the vector field $\hat{\Sigma}$
to the vector field $\hat{\sigma}$.
An adequate solution theory for these electro-magneto static problems
can be found in \cite{picardpotential,picardboundaryelectro,saranenelectromagnetostatic,picardlowfreqmax,picarddeco}.
In fact, $v_{\sigma}=\pi\hat{\sigma}$ is the Helmholtz projection $\pi$ of $\hat{\sigma}$ 
onto solenoidal vector fields with homogeneous normal boundary condition.
Generally, $v_{\sigma}\not\in\mathsf{V}_{\mathsf{n},\sigma}(\om)$ 
and thus $\mathsf{V}_{\mathsf{n},\sigma}(\om)=\emptyset$, i.e., $c_{\mathsf{g}}=+\infty$, 
is possible even for $\co$-domains\footnote{In \cite[Lemma 4]{desvillettesvillanikornnormal}
$v_{\sigma}$ is found by solving the Neumann problem,
$\Delta\varphi=0$ in $\om$, $\na\varphi\cdot\nu=-\hat{\Sigma}\,\nu$ on $\ga$,
and setting $v_\sigma=\na\varphi+\hat\Sigma$. 
But in order to guarantee $v_\sigma\in\hoom$ one needs to have $\varphi\in\htom$, 
which itself is only ensured if $\om$ is $\ct$ or convex. 
Moreover, as pointed out above it seems to be unclear wether
\cite[(10), (13)]{desvillettesvillanikornnormal}, i.e., \eqref{maxwellregvillani}, hold for general $\co$-domains.}.
On the other hand, if $\om$ is $\ct$ or convex, the above mentioned regularity theory for Maxwell's equations
shows $v_{\sigma}\in\mathsf{V}_{\mathsf{n},\sigma}(\om)$. Moreover, if $\om$ is convex or simply connected and $\ct$, 
there are even no Neumann fields, which implies in these cases 
the uniqueness of the solution $v_{\sigma}$ and we simply have
$$c_{\mathsf{g}}
=\frac{1}{|\om|}\inf_{|\sigma|=1}\normltom{\sym\na v_{\sigma}}^2.$$
As announced in the introduction we now show that
$c_{\mathsf{g}}(\om_{\mathsf{p}})=\oh$ holds
for any bounded and convex polyhedron $\om_{\mathsf{p}}\subset\rN$.
For every $\sigma\in S^{(N-1)N/2-1}$
problem \eqref{Maxwellproblem} has a unique solution $v_{\sigma}\in\hocn(\om_{\mathsf{p}})$, i.e.
$v_{\sigma}\in\mathsf{V}_{\mathsf{n},\sigma}(\om_{\mathsf{p}})$,
(by regularity for static Maxwell's equations in convex domains,
see e.g. \cite[Theorem 3.1]{saranenineqfried} or 
\cite[Theorem 2.17]{amrouchebernardidaugegiraultvectorpot} for the case $N=3$) with
$$\norm{\rot v_{\sigma}}_{\lt(\om_{\mathsf{p}})}^2
=\norm{\sigma}_{\lt(\om_{\mathsf{p}})}^2
=|\om_{\mathsf{p}}|.$$
On the other hand, by Corollary \ref{concavepolyhedroncor} and Theorem \ref{maintheo} we also have
$$\norm{\rot v_{\sigma}}_{\lt(\om_{\mathsf{p}})}^2
=\norm{\na v_{\sigma}}_{\lt(\om_{\mathsf{p}})}^2
=2\norm{\dev\sym\na v_{\sigma}}_{\lt(\om_{\mathsf{p}})}^2
=2\norm{\sym\na v_{\sigma}}_{\lt(\om_{\mathsf{p}})}^2$$
and hence
$$c_{\mathsf{g}}
=\frac{1}{|\om_{\mathsf{p}}|}\inf_{|\sigma|=1}\norm{\sym\na v_{\sigma}}_{\lt(\om_{\mathsf{p}})}^2
=\foh.$$

\section{Some More Estimates on the Gradient}

In this section we shall combine some more pointwise formulas and estimates on matrices and Jacobians with
the integration formula from Proposition \ref{GrisvardPro} in order to get some more equalities and estimates 
on the norm of gradients.

\subsection{Matrices}

Let us note a few simple and well known facts about matrices and Jacobians
extending the formulas presented in Section \ref{sec:prelim}.
The pointwise orthogonal sums
$$A=\dev A\oplus\id_{A},\quad
A=\sym A\oplus\skw A,\quad
\sym A=\dev\sym A\oplus\id_{A}$$
translate to the pointwise equations
$$\norm{A}^2
=\norm{\dev A}^2
+\frac{1}{N}\norm{\tr A}^2,\quad
\norm{A}^2
=\norm{\sym A}^2
+\norm{\skw A}^2,\quad
\norm{\sym A}^2
=\norm{\dev\sym A}^2
+\frac{1}{N}\norm{\tr A}^2$$
and the pointwise estimates
$$\norm{\dev A},
\frac{1}{\sqrt{N}}\norm{\tr A},
\norm{\sym A},
\norm{\skw A}
\leq\norm{A}.$$
For $A=\na v$ with $v\in\hoom$ we see pointwise a.e.
\begin{align}
\mylabel{appskewrot}
\norm{\skw\na v}^2
=\foh\norm{\rot v}^2,\quad
\tr\na v=\div v
\end{align}
and
\begin{align*}
\norm{\na v}^2
&=\norm{\sym\na v}^2
+\norm{\skw\na v}^2
=\norm{\sym\na v}^2
+\foh\norm{\rot v}^2,\\
\norm{\na v}^2
&=\norm{\dev\sym\na v}^2
+\frac{1}{N}\norm{\div v}^2
+\norm{\skw\na v}^2
=\norm{\dev\sym\na v}^2
+\frac{1}{N}\norm{\div v}^2
+\foh\norm{\rot v}^2.
\end{align*}
Especially, we see
$$\norm{\div v}^2
+\norm{\rot v}^2
\leq N\norm{\na v}^2.$$
Moreover, by
$$2\norm{\sym\!/\!\skw\na v}^2
=\foh\norm{\na v\pm(\na v)^{\top}}^2
=\norm{\na v}^2
\pm\scp{\na v}{(\na v)^{\top}}$$
we get
\begin{align}
\mylabel{appsymgrad}
\norm{\na v}^2
&=2\norm{\sym\na v}^2
-\scp{\na v}{(\na v)^{\top}},\\
\mylabel{appskewgrad}
\norm{\na v}^2
&=2\norm{\skw\na v}^2
+\scp{\na v}{(\na v)^{\top}}
=\norm{\rot v}^2
+\scp{\na v}{(\na v)^{\top}}.
\end{align}

\subsection{Integration by Parts}

Defining
$$I_{\mathsf{b}}
:=\int_{\ga_1}\big(\ncomp{v}\div_\ga\tcomp{v}
-\tcomp{v}\cdot\na_\ga \ncomp{v}\big),\quad
I_{\mathsf{c}}
:=\int_{\ga_1}\big(\div\nu\,|v_{\mathsf{n}}|^2
+((\na\nu)\,v_{\mathsf{t}})\cdot v_{\mathsf{t}}\big)$$
Proposition \ref{GrisvardPro} as well as
\eqref{appsymgrad} and \eqref{appskewgrad}, \eqref{appskewrot}
show:

\begin{lem}[integration by parts]
\mylabel{appGrisvardLem}
Let $\om\subset\rN$ be piecewise $\ct$. Then for all $v\in\cicomb$
\begin{align*}
\normltom{\na v}^2
-I_{\mathsf{b}}
-I_{\mathsf{c}}
&=2\normltom{\sym\na v}^2
-\normltom{\div v}^2
=2\normltom{\dev\sym\na v}^2
+\frac{2-N}{N}\normltom{\div v}^2,\\
\normltom{\na v}^2
+I_{\mathsf{b}}
+I_{\mathsf{c}}
&=2\normltom{\skw\na v}^2
+\normltom{\div v}^2
=\normltom{\rot v}^2
+\normltom{\div v}^2,
\end{align*}
which extend by continuity to all $v\in\htom$.
For $v\in\cictnom$ the integral $I_{\mathsf{b}}$ containing the boundary differential operators vanishes
and the formulas (without $I_{\mathsf{b}}$) extend by continuity to all $v\in\hoctnom$.
\end{lem}

For $N=2,3$ these results have already been presented in 
\cite[Lemma 2.2, Theorem 2.3, Remark 2.4]{costabeldaugemaxwelllameeigenvaluespolyhedra}.

\subsection{Gradient Estimates}

By Lemma \ref{appGrisvardLem} we get for all $v\in\hoctnom$
\begin{align*}
I_{\mathsf{c}}
=\begin{cases}
\diss\normltom{\na v}^2
-2\normltom{\dev\sym\na v}^2
-\frac{2-N}{N}\normltom{\div v}^2\\
\normltom{\rot v}^2
+\normltom{\div v}^2
-\normltom{\na v}^2
\end{cases},\quad
I_{\mathsf{c}}
\leq c\int_{\ga_1}|v|^2,
\end{align*}
where $c>0$ just depends on the derivatives of $\nu$ and $\ga_1$.
In combination with \cite[Theorem 1.5.1.10]{grisvardbook} we obtain:

\begin{cor}
\mylabel{appGrisvardCor}
Let $\om\subset\rN$ be piecewise $\ct$. Then there exists $c>0$,
such that for all $v\in\hoctnom$ and for all $\eps>0$
\begin{align*}
(1-\eps)\normltom{\na v}^2
&\leq2\normltom{\dev\sym\na v}^2
+\frac{2-N}{N}\normltom{\div v}^2
+\frac{c}{\eps}\normltom{v}^2,\\
(1-\eps)\normltom{\na v}^2
&\leq\normltom{\rot v}^2
+\normltom{\div v}^2
+\frac{c}{\eps}\normltom{v}^2.
\end{align*}
\end{cor}

The latter lemma and corollary clearly show that Korn's inequalities and the Maxwell gradient estimate
\eqref{maxwellregvillani} share the same origin. 
We can also get rid of the $\ltom$-norm of $v$ on the right hand sides.
Let us focus on the second inequality and assume that 
$\om$ is a bounded and piecewise $\ct$-domain. We introduce
$$\mathsf{V}_{\mathsf{t,n},0}(\om):=\set{v\in\hoctnom}{\div v=0\,\wedge\,\rot v=0},$$
which is a finite dimensional (and hence closed) subspace of $\ltom$
since its unit ball is compact by Corollary \ref{appGrisvardCor} and Rellich's selection theorem.

\begin{cor}[Gaffney's inequality]
\mylabel{appGrisvardCorwithoutltnorm}
Let $\om\subset\rN$ be a bounded and piecewise $\ct$-domain. Then there exists $c>0$,
such that for all $v\in\hoctnom$ there exists $n_{v}\in\mathsf{V}_{\mathsf{t,n},0}(\om)$ with
$$\norm{v-n_{v}}_{\hoom}
\leq c\big(\normltom{\rot v}+\normltom{\div v}\big)$$
and $n_{v}$ is the $\ltom$-orthonormal projection of $v$ onto $\mathsf{V}_{\mathsf{t,n},0}(\om)$.
\end{cor}

\begin{proof}
Since $v-n_{v}\in\hoctnom\cap\mathsf{V}_{\mathsf{t,n},0}(\om)^{\bot}$
as well as $\rot(v-n_{v})=\rot v$ and $\div(v-n_{v})=\div v$ 
it is sufficient to show
\begin{align}
\mylabel{appGrisvardCorwithoutltnormineqone}
\exists\,c&>0&
\forall\,v&\in\hoctnom\cap\mathsf{V}_{\mathsf{t,n},0}(\om)^{\bot}&
\norm{v}_{\hoom}
&\leq c\big(\normltom{\rot v}+\normltom{\div v}\big).
\end{align}
If \eqref{appGrisvardCorwithoutltnormineqone} is wrong, there exists a sequence 
$(v_{n})\subset\hoctnom\cap\mathsf{V}_{\mathsf{t,n},0}(\om)^{\bot}$ with
$$\norm{v_{n}}_{\hoom}=1,\quad
\normltom{\rot v_{n}}+\normltom{\div v_{n}}\to0.$$
As $(v_{n})$ is bounded in $\hoom$ there exists a subsequence $(v_{\pi n})$
converging to some $v$ in $\ltom$ by Rellich's selection theorem. 
By Corollary \ref{appGrisvardCor}, $(v_{\pi n})$ is a Cauchy sequence in $\hoom$ and thus 
$$v_{\pi n}\to v\in\hoctnom\cap\mathsf{V}_{\mathsf{t,n},0}(\om)^{\bot}\qtext{in}\hoom.$$
Since $v$ belongs to $\mathsf{V}_{\mathsf{t,n},0}(\om)$ as well, we have $v=0$
in contradiction to $1=\norm{v_{n}}_{\hoom}\to\norm{v}_{\hoom}$, which proves \eqref{appGrisvardCorwithoutltnormineqone}.
\end{proof}

\begin{rem}
\mylabel{appGrisvardCorwithoutltnormrem}
As in Remark \ref{maintheorem} and Corollaries \ref{maintheocor} and \ref{maintheocorrotdiv},
there are also versions of Corollary \ref{appGrisvardCorwithoutltnorm} for the case
of e.g. a piecewise $\ct$ exterior domain using polynomially weighted Sobolev spaces.
\end{rem}

%%%%%%%%%%%%%%%%%%%%%%%%%%%%%%%%%%%%%%%%%%%%%%%%%%%%%%%%%%%%%%%%%%%%%%%%%%%%%%%%

%%%%%%%%%%%%%%%%
% bibliography %
%%%%%%%%%%%%%%%%

\bibliographystyle{plain} 
\ifthenelse{\equal{\person}{paule}}
{\bibliography{paule,Litbank}}
{\bibliography{/Users/sebastianbauer/Dropbox/Uni/paper/paule,/Users/sebastianbauer/Dropbox/Uni/paper/Litbank}}

%%%%%%%%%%%%%%%%%%%%%%%%%%%%%%%%%%%%%%%%%%%%%%%%%%%%%%%%%%%%%%%%%%%%%%%%%%%%%%%%
\end{document}